\documentclass[12pt,reqno]{amsart}
\usepackage{amsmath,amssymb,cite,geometry,bm,easybmat,enumerate}
\usepackage[colorlinks,citecolor=blue,linkcolor=blue,urlcolor=blue]{hyperref}
\newtheorem{theorem}{Theorem}[section]

\newtheorem{proposition}[theorem]{Proposition}

\numberwithin{equation}{section}
\allowdisplaybreaks[4]

\begin{document}

\title[Boundary determination of the Riemannian metric]{Boundary determination of the Riemannian metric by the elastic Dirichlet-to-Neumann map}

\author{Xiaoming Tan}
\address{Beijing International Center for Mathematical Research, Peking University, Beijing 100871, China}
\email{tanxm@pku.edu.cn}


\keywords{Boundary determination; Riemannian metric; Inverse problems; elastic Dirichlet-to-Neumann map; Pseudodifferential operators.\\
{\bf -----------------}\\
{\it Email address}: tanxm@pku.edu.cn\\
Beijing International Center for Mathematical Research, Peking University, Beijing 100871, China}

\begin{abstract}
    For a compact connected Riemannian manifold with smooth boundary, by computing the full symbol of the elastic Dirichlet-to-Neumann map, we prove that the elastic Dirichlet-to-Neumann map can uniquely determine the partial derivatives of all orders of the Riemannian metric on the boundary of the manifold.
\end{abstract}

\maketitle 

\section{Introduction}

\vspace{5mm}

Let $(M,g)$ be a compact connected Riemannian manifold of dimension $n$ with smooth boundary $\partial M$. In this paper we consider $M$ as an inhomogeneous, isotropic, elastic medium. Assume that the Lam\'{e} coefficients $\lambda,\,\mu \in C^{\infty}(\bar{M})$ of this elastic medium are known and satisfy
\begin{align}\label{0.3}
    \mu>0, \quad \lambda + \mu \geqslant 0.
\end{align}

In the local coordinates $\{x_j\}_{j=1}^n$, we denote by $\bigl\{\frac{\partial}{\partial x_j}\bigr\}_{j=1}^n$ and $\{dx_j\}_{j=1}^n$, respectively, the natural basis for the tangent space $T_x M$ and the cotangent space $T_x^{*} M$ at the point $x\in M$. In what follows, we will use the Einstein summation convention. The Greek indices run from 1 to $n-1$, whereas the Roman indices run from 1 to $n$, unless otherwise specified. Then, the Riemannian metric $g$ is given by $g = g_{jk} \,dx_j\otimes dx_k$. Let $\nabla_j=\nabla_{\frac{\partial}{\partial x_j}}$ be the covariant derivative with respect to $\frac{\partial}{\partial x_j}$ and $\nabla^j= g^{jk} \nabla_k$, where $[g^{jk}]=[g_{jk}]^{-1}$. 

\subsection{Lam\'{e} operator}

The Lam\'{e} operator $\mathcal{L}_{\lambda,\mu}$ with variable coefficients $\lambda$ and $\mu$ on the Riemannian manifold is given by, for smooth displacement vector field $\bm{u}=u^j\frac{\partial}{\partial x_j}$ (see \cite{TanLiu23}),
\begin{align}\label{1.1}
    \mathcal{L}_{\lambda,\mu} \bm{u} 
    &:= \mu \Delta^{}_{B} \bm{u} + (\lambda + \mu)\operatorname{grad}\operatorname{div} \bm{u} + \mu \operatorname{Ric}(\bm{u}) \notag\\
    &\qquad + (\operatorname{grad} \lambda) \operatorname{div} \bm{u} + (S\bm{u})(\operatorname{grad} \mu), 
\end{align} 
where the Bochner Laplacian $(\Delta^{}_{B}\bm{u})^j := \nabla^k \nabla_k u^j$, $\operatorname{grad}$ and $\operatorname{div}$ are, respectively, the gradient and divergence operators, the strain tensor $S$ is defined by (see \cite[p.\,562]{Taylor11.3})
\begin{align}\label{1.5}
    (S\bm{u})^j_k := \nabla^j u_k + \nabla_k u^j,\quad u_k:=g_{kl}u^l.
\end{align}
Here $\operatorname{Ric}(\bm{u})^j = g^{jk}R_{kl} u^l$, where $R_{kl}$ are the components of Ricci tensor of the manifold, i.e.,
\begin{align}\label{0.03}
    R_{kl}=\frac{\partial \Gamma^{j}_{kl}}{\partial x_j} - \frac{\partial \Gamma^{j}_{jl}}{\partial x_k} + \Gamma^{j}_{jm} \Gamma^{m}_{kl} - \Gamma^{j}_{km} \Gamma^{m}_{jl},
\end{align}
the Christoffel symbols
\begin{align*}
    \Gamma^{j}_{kl} = \frac{1}{2} g^{jm} \Bigl(\frac{\partial g_{km}}{\partial x_l} + \frac{\partial g_{lm}}{\partial x_k} - \frac{\partial g_{kl}}{\partial x_m}\Bigr).
\end{align*}

\subsection{Elastic Dirichlet-to-Neumann map}

We consider the following Dirichlet boundary value problem for the Lam\'{e} system on the Riemannian manifold:
\begin{align}\label{1.3}
    \begin{cases}
        \mathcal{L}_{\lambda,\mu} \bm{u} = 0  & \text{in}\ M, \\
        \bm{u}= \bm{f} \quad & \text{on}\ \partial M.
    \end{cases}
\end{align}
For any displacement $\bm{f}\in[H^{1/2}(\partial M)]^n$ on the boundary, by the theory of elliptic operators, there is a unique solution $\bm{u}\in[H^1(M)]^n$ solving the above problem \eqref{1.3}. Therefore, the elastic Dirichlet-to-Neumann map (also called the displacement-to-traction map) $\Lambda_{\lambda,\mu}:[H^{1/2}(\partial M)]^n \to [H^{-1/2}(\partial M)]^n$ associated with the operator $\mathcal{L}_{\lambda,\mu}$ is defined by (see \cite{TanLiu23})
\begin{align}\label{1.8}
    \Lambda_{\lambda,\mu}(\bm{f}) := \tau(\nu) = \lambda (\operatorname{div} \bm{u})\nu + \mu (S \bm{u})\nu \quad \text{on}\ \partial M,
\end{align}
where $\nu$ is the outward unit normal vector to the boundary $\partial M$. Physically, $\tau(\nu)$ is called the Neumann boundary condition (or free boundary condition, or traction boundary condition) of \eqref{1.3}. Roughly speaking, the elastic Dirichlet-to-Neumann map $\Lambda_{\lambda,\mu}$ sends the displacement at the boundary to the corresponding normal component of the stress (i.e., the traction) at the boundary (see \cite{NakaUhlm94}). It is clear that the elastic Dirichlet-to-Neumann map $\Lambda_{\lambda,\mu}$ is an elliptic, self-adjoint pseudodifferential operator of order one defined on the boundary $\partial M$.

In \cite{LeeUhlm89}, the authors proved that the Dirichlet-to-Neumann map uniquely determine the real-analytic Riemannian metric. In \cite{LassUhlm01}, the authors studied the inverse problem of determining a Riemannian manifold from the boundary data of harmonic functions, this extend the results in \cite{LeeUhlm89}. Moreover, \cite{LassTaylUhlm03} considered the case of complete Riemannian manifold.
In \cite{Liu19,Liu19.2}, the author proved that the elastic Dirichlet-to-Neumann map with constant coefficients and the electromagnetic Dirichlet-to-Neumann map can uniquely determine the real-analytic Riemannian metric and parameters. In \cite{TanLiu23}, the authors gave an explicit expression for the full symbol of the elastic Dirichlet-to-Neumann map with variable coefficients and proved that the elastic Dirichlet-to-Neumann map uniquely determines the Lam\'{e} coefficients.

\vspace{2mm}

The main result of this paper is the following theorem.

\begin{theorem}\label{thm1.1}
    Let $(M,g)$ be a compact connected Riemannian manifold of dimension $n$ with smooth boundary $\partial M$. Assume that the Lam\'{e} coefficients $\lambda,\,\mu \in C^{\infty}(\bar{M})$ satisfy $\mu>0$ and $\lambda + \mu \geqslant 0$. Then, the elastic Dirichlet-to-Neumann map $\Lambda_{\lambda,\mu}$ uniquely determines the partial derivatives of all orders of the Riemannian metric $\frac{\partial^{|J|} g^{\alpha\beta}}{\partial x^J}$ on the boundary $\partial M$ for all multi-indices $J$.
\end{theorem}

This paper is organized as follows. In Section \ref{s2}, we give the symbols of some pseudodifferential operators. In Section \ref{s3}, we prove the main result by the full symbol of the elastic Dirichlet-to-Neumann map.

\vspace{5mm}

\section{Symbols of the pseudodifferential operators}\label{s2}

\vspace{5mm}

For the sake of simplicity, we denote by $i=\sqrt{-1}$, $\xi^{\prime}=(\xi_1,\dots,\xi_{n-1})$, $\xi^\alpha=g^{\alpha\beta}\xi_\beta$, $|\xi^{\prime}|=\sqrt{\xi^\alpha\xi_\alpha}$, $I_n$ the $n\times n$ identity matrix, and
\begin{align*}
    \begin{bmatrix}
        [a^\alpha_\beta] & [a^{\alpha}] \\
        [a_{\beta}] & a^{n}_{n}
    \end{bmatrix}
    :=\begin{bmatrix}
        \begin{BMAT}{ccc.c}{ccc.c}
            a^1_{1} & \dots & a^1_{n-1} \ & a^1_{n} \\
            \vdots & \ddots & \vdots\ & \vdots \\
            a^{n-1}_1 & \dots & a^{n-1}_{n-1}\ &\ a^{n-1}_{n}\\
            a^{n}_1 & \dots & a^{n}_{n-1}\ & a^{n}_{n}
        \end{BMAT}
    \end{bmatrix},
\end{align*}
where $ 1 \leqslant \alpha,\beta \leqslant n-1$.

In the previous work \cite{TanLiu23}, for the Lam\'{e} operator $\mathcal{L}_{\lambda,\mu}$ with variable coefficients $\lambda$ and $\mu$, we had deduced that
\begin{align}\label{3.07}
    A^{-1} \mathcal{L}_{\lambda,\mu} = I_{n}\frac{\partial^2 }{\partial x_n^2} + B \frac{\partial }{\partial x_n} + C,
\end{align}
where
\begin{align}\label{3.2}
        A=
        \begin{bmatrix}
            \mu I_{n-1} &0  \\
            0& \lambda+2\mu 
        \end{bmatrix},
\end{align}
$B=B_1+B_0$, $C =C_2+C_1+C_0$, and
\begin{align*}
    &B_1=(\lambda+\mu)
    \begin{bmatrix}
        0 & \displaystyle \frac{1}{\mu}\Big[g^{\alpha\beta}\frac{\partial}{\partial x_{\beta}}\Bigr]\\
        \displaystyle \frac{1}{\lambda+2\mu}\Big[\frac{\partial}{\partial x_{\beta}}\Bigr] & 0 
    \end{bmatrix},\\
    &B_0=
    \begin{bmatrix}
        \Gamma^\alpha_{\alpha n} I_{n-1}+2[\Gamma^\alpha_{\beta n}] & 0 \\[2mm]
        \displaystyle \frac{\lambda+\mu}{\lambda+2\mu}[\Gamma^\alpha_{\alpha\beta}] & \Gamma^\alpha_{\alpha n}
    \end{bmatrix}
    +\begin{bmatrix}
        \displaystyle \frac{1}{\mu} \frac{\partial \mu}{\partial x_{n}} I_{n-1} &  \displaystyle \frac{1}{\mu} [\nabla^\alpha \lambda] \\[4mm]
        \displaystyle \frac{1}{\lambda+2\mu} \Big[\frac{\partial \mu}{\partial x_{\beta}}\Big] &  \displaystyle \frac{1}{\lambda+2\mu} \frac{\partial (\lambda+2\mu)}{\partial x_{n}}
    \end{bmatrix},\\
    &C_2=
    \begin{bmatrix}
        \displaystyle \Big(g^{\alpha\beta}\frac{\partial^2 }{\partial x_\alpha \partial x_\beta}\Big)I_{n-1} + \frac{\lambda+\mu}{\mu}\Big[g^{\alpha\gamma}\frac{\partial^2 }{\partial x_\gamma \partial x_\beta}\Bigr] & 0 \\
        0 &  \displaystyle \frac{\mu}{\lambda+2\mu}g^{\alpha\beta}\frac{\partial^2 }{\partial x_\alpha \partial x_\beta} 
    \end{bmatrix},\\
    &C_1=
    \begin{bmatrix}
        \displaystyle \Big(\Big(g^{\alpha\beta}\Gamma^\gamma_{\alpha\gamma}+\frac{\partial g^{\alpha\beta}}{\partial x_{\alpha}}\Big) \frac{\partial }{\partial x_{\beta}}\Big) I_{n-1}  & 0 \\
        0 &  \displaystyle \frac{\mu}{\lambda+2\mu}\Big(g^{\alpha\beta}\Gamma^\gamma_{\alpha\gamma}+\frac{\partial g^{\alpha\beta}}{\partial x_{\alpha}}\Big)\frac{\partial }{\partial x_{\beta}}
    \end{bmatrix}\\
    &\ + \frac{\lambda+\mu}{\mu}
    \begin{bmatrix}
        \displaystyle \Big[g^{\alpha\gamma}\Gamma^\rho_{\rho\beta}\frac{\partial }{\partial x_{\gamma}}\Big] &  \displaystyle \Big[g^{\alpha\gamma}\Gamma^\rho_{\rho n}\frac{\partial }{\partial x_{\gamma}}\Big] \\
        0 & 0
    \end{bmatrix}
    +2
    \begin{bmatrix}
        \displaystyle \Big[g^{\gamma\rho}\Gamma^\alpha_{\rho\beta}\frac{\partial }{\partial x_{\gamma}}\Big] &  \displaystyle \Big[g^{\gamma\rho}\Gamma^\alpha_{\rho n}\frac{\partial }{\partial x_{\gamma}}\Big] \\[4mm]
        \displaystyle \frac{\mu}{\lambda+2\mu}\Big[g^{\gamma\rho}\Gamma^n_{\rho\beta}\frac{\partial }{\partial x_{\gamma}}\Big] & 0
    \end{bmatrix}\\
    &\ +
    \begin{bmatrix}
        \displaystyle \frac{1}{\mu}\Big(\nabla^\alpha \mu \frac{\partial }{\partial x_{\alpha}}\Big) I_{n-1} + \frac{1}{\mu} \Big[\nabla^\alpha \lambda \frac{\partial }{\partial x_{\beta}} + g^{\alpha\gamma} \frac{\partial \mu}{\partial x_{\beta}} \frac{\partial }{\partial x_{\gamma}}\Big] &  \displaystyle \frac{1}{\mu} \frac{\partial \mu}{\partial x_{n}}\Big[g^{\alpha\beta} \frac{\partial }{\partial x_{\beta}}\Big] \\[4mm]
        \displaystyle \frac{1}{\lambda+2\mu} \frac{\partial \lambda}{\partial x_{n}} \Big[\frac{\partial }{\partial x_{\beta}}\Big] &  \displaystyle \frac{1}{\lambda+2\mu} \nabla^{\alpha}\mu \frac{\partial }{\partial x_{\alpha}}
    \end{bmatrix},\\
    &C_0 =(\lambda+\mu)
    \begin{bmatrix}
        \displaystyle \frac{1}{\mu}\Big[g^{\alpha\gamma}\frac{\partial \Gamma^\rho_{\rho\beta}}{\partial x_\gamma}\Big] &  \displaystyle \frac{1}{\mu}\Big[g^{\alpha\gamma}\frac{\partial \Gamma^\rho_{\rho n}}{\partial x_\gamma}\Big] \\[4mm]
        \displaystyle \frac{1}{\lambda+2\mu}\Big[\frac{\partial \Gamma^\alpha_{\alpha\beta}}{\partial x_n}\Big] &  \displaystyle \frac{1}{\lambda+2\mu}\frac{\partial \Gamma^\alpha_{\alpha n}}{\partial x_n}
    \end{bmatrix}
    +
    \begin{bmatrix}
        \displaystyle \Big[g^{ml}\frac{\partial \Gamma^\alpha_{ml}}{\partial x_\beta}\Big] &  \displaystyle \Big[g^{ml}\frac{\partial \Gamma^\alpha_{ml}}{\partial x_n}\Big] \\[4mm]
        \displaystyle \frac{\mu}{\lambda+2\mu}\Big[g^{ml}\frac{\partial \Gamma^n_{ml}}{\partial x_\beta}\Big] &  \displaystyle \frac{\mu}{\lambda+2\mu}g^{ml}\frac{\partial \Gamma^n_{ml}}{\partial x_n}
    \end{bmatrix}\\
    &\ +
    \begin{bmatrix}
        \displaystyle \frac{1}{\mu} \Big[\nabla^\alpha \lambda \Gamma^\gamma_{\beta\gamma} - \frac{\partial \mu}{\partial x_{\gamma}} \frac{\partial g^{\alpha\gamma}}{\partial x_{\beta}}\Big] &  \displaystyle \frac{1}{\mu} \Big[\nabla^\alpha \lambda \Gamma^\beta_{\beta n} - \frac{\partial \mu}{\partial x_{\beta}} \frac{\partial g^{\alpha\beta}}{\partial x_{n}}\Big] \\[4mm]
        \displaystyle \frac{1}{\lambda+2\mu} \frac{\partial \lambda}{\partial x_{n}} [\Gamma^\alpha_{\alpha\beta}] &  \displaystyle \frac{1}{\lambda+2\mu} \frac{\partial \lambda}{\partial x_{n}} \Gamma^\alpha_{\alpha n}
    \end{bmatrix}.
\end{align*}

Let 
    \begin{align*}
        b(x,\xi^{\prime})=b_1(x,\xi^{\prime}) + b_0(x,\xi^{\prime})
    \end{align*}
    and
    \begin{align*}
        c(x,\xi^{\prime}) = c_2(x,\xi^{\prime}) + c_1(x,\xi^{\prime}) + c_0(x,\xi^{\prime})
    \end{align*}
    be the full symbols of $B$ and $C$, respectively, where $b_j(x,\xi^{\prime})$ and $c_j(x,\xi^{\prime})$ are homogeneous of degree $j$ in $\xi^{\prime}$. We simply write
    \begin{align*}
        \xi^{\alpha}=g^{\alpha\beta}\xi_{\beta},\quad |\xi^{\prime}|^2=\xi^{\alpha}\xi_{\alpha}=g^{\alpha\beta}\xi_\alpha\xi_\beta.
    \end{align*}
    Thus, we obtain (see \cite{TanLiu23})
    \begin{align}
        &\label{5.9} b_1(x,\xi^{\prime})=i(\lambda+\mu)
        \begin{bmatrix}
            0 &  \displaystyle \frac{1}{\mu}[\xi^\alpha] \\
            \displaystyle \frac{1}{\lambda+2\mu}[\xi_\beta] & 0
        \end{bmatrix},\\
        &\label{5.10} b_0(x,\xi^{\prime})=B_0,\\
        &\label{5.11} c_2(x,\xi^{\prime})= -
        \begin{bmatrix}
            \displaystyle |\xi^{\prime}|^2 I_{n-1} + \frac{\lambda+\mu}{\mu}[\xi^\alpha\xi_\beta] & 0 \\
            0 &  \displaystyle \frac{\mu}{\lambda+2\mu}|\xi^{\prime}|^2
        \end{bmatrix},\\
        &\label{5.12} c_1(x,\xi^{\prime})= i
        \begin{bmatrix}
            \displaystyle \Big(\xi^\alpha\Gamma^\beta_{\alpha\beta}+\frac{\partial \xi^\alpha}{\partial x_{\alpha}}\Big) I_{n-1} & 0 \\
            0 &  \displaystyle \frac{\mu}{\lambda+2\mu}\Big(\xi^\alpha\Gamma^\beta_{\alpha\beta}+\frac{\partial \xi^\alpha}{\partial x_{\alpha}}\Big)
        \end{bmatrix} \notag\\
        &\quad +\frac{i(\lambda+\mu)}{\mu}
        \begin{bmatrix}
            \displaystyle [\xi^\alpha\Gamma^\gamma_{\gamma\beta}] &  \displaystyle \Gamma^\beta_{\beta n}[\xi^\alpha] \\
            0 & 0
        \end{bmatrix} + 2i
        \begin{bmatrix}
            \displaystyle [\xi^\gamma\Gamma^\alpha_{\gamma\beta}] &  \displaystyle [\xi^\gamma\Gamma^\alpha_{\gamma n}] \\[2mm]
            \displaystyle \frac{\mu}{\lambda+2\mu}[\xi^\gamma\Gamma^n_{\gamma\beta}] & 0 
        \end{bmatrix} \notag\\
        &\quad + i
        \begin{bmatrix}
            \displaystyle \frac{1}{\mu} (\xi_{\alpha} \nabla^\alpha \mu) I_{n-1} + \frac{1}{\mu} \Big[\xi_{\beta}\nabla^\alpha \lambda  + \xi^\alpha \frac{\partial \mu}{\partial x_{\beta}}\Big] &  \displaystyle \frac{1}{\mu} \frac{\partial \mu}{\partial x_{n}} [\xi^\alpha] \\[4mm]
            \displaystyle \frac{1}{\lambda+2\mu} \frac{\partial \lambda}{\partial x_{n}} [\xi_{\beta}] &  \displaystyle \frac{1}{\lambda+2\mu} \xi_{\alpha} \nabla^{\alpha}\mu
        \end{bmatrix},\\
        &\label{5.13} c_0(x,\xi^{\prime}) =C_0.
    \end{align} 

For the convenience of stating the following proposition, we define 
\begin{align}
    \label{2.11} E_1&:=i\sum_\alpha\frac{\partial (q_1-b_1)}{\partial \xi_\alpha}\frac{\partial q_1}{\partial x_\alpha}+b_0q_1+\frac{\partial q_1}{\partial x_n} - c_1,\\
    \label{3.05} E_0&:=i\sum_\alpha\Bigl(\frac{\partial (q_1-b_1)}{\partial \xi_\alpha}\frac{\partial q_0}{\partial x_\alpha}+\frac{\partial q_0}{\partial \xi_\alpha}\frac{\partial q_1}{\partial x_\alpha}\Bigr)+\frac{1}{2}\sum_{\alpha,\beta}\frac{\partial^2q_1}{\partial \xi_\alpha \partial\xi_\beta}\frac{\partial^2q_1}{\partial x_\alpha \partial x_\beta} \notag\\
    &\quad -q_0^2 +b_0q_0 +\frac{\partial q_0}{\partial x_n} - c_0,\\
    \label{4.1} E_{-m}&:= b_0q_{-m}+\frac{\partial q_{-m}}{\partial x_n} - i\sum_\alpha\frac{\partial b_1}{\partial \xi_\alpha}\frac{\partial q_{-m}}{\partial x_\alpha} - \sum_{\substack{-m \leqslant j,k \leqslant 1 \\ |J| = j + k + m}} \frac{(-i)^{|J|}}{J !} \partial_{\xi^{\prime}}^{J} q_j\, \partial_{x^\prime}^{J} q_k
\end{align}
for $m \geqslant 1$, where $q_j=q_j(x,\xi^{\prime})$, $b_j=b_j(x,\xi^{\prime})$, and $c_j=c_j(x,\xi^{\prime})$.

We have the following two results.

\begin{proposition}[see \cite{TanLiu23}]\label{prop3.1}
    Let $Q(x,\partial_{x^\prime})$ be a pseudodifferential operator of order one in $x^\prime$ depending smoothly on $x_n$ such that
    \begin{align*}
        A^{-1}\mathcal{L}_{\lambda,\mu}
        = \Bigl(I_{n}\frac{\partial }{\partial x_n} + B - Q\Bigr)\Bigl(I_{n}\frac{\partial }{\partial x_n} + Q\Bigr)
    \end{align*}
    modulo a smoothing operator. Let $q(x,\xi^{\prime}) \sim \sum_{j\leqslant 1} q_j(x,\xi^{\prime})$ be the full symbol of $Q$, where $q_j(x,\xi^{\prime})\in \mathcal{S}^{j}$ are homogeneous of degree $j$ in $\xi^{\prime}\in \mathbb{R}^{n-1}$. Then, in boundary normal coordinates,
    \begin{align}
        q_1(x,\xi^{\prime})&=|\xi^{\prime}|I_n + \frac{\lambda+\mu}{\lambda+3\mu}F_1, \label{3.9}\\
        q_{-m-1}(x,\xi^{\prime})&=\frac{1}{2|\xi^{\prime}|}E_{-m} - \frac{\lambda+\mu}{4(\lambda+3\mu)|\xi^{\prime}|^2}(F_2E_{-m}+E_{-m}F_1) \notag\\ 
        &\quad + \frac{(\lambda+\mu)^2}{4(\lambda+3\mu)^2|\xi^{\prime}|^3}F_2E_{-m}F_1, \quad m\geqslant -1, \label{3.1.1}
    \end{align}
    where $E_{-m}$, $m\geqslant -1$, are given by \eqref{2.11}--\eqref{4.1}, and 
    \begin{align}
        \label{5.24} F_1&=
        \begin{bmatrix}
            \displaystyle \frac{1}{|\xi^{\prime}|}[\xi^\alpha\xi_\beta] & i[\xi^\alpha] \\[4mm]
            i[\xi_\beta] & -|\xi^{\prime}|
        \end{bmatrix},\\
        \label{5.25} F_2&=
        \begin{bmatrix}
            \displaystyle \frac{1}{|\xi^{\prime}|}[\xi^\alpha\xi_\beta] &  \displaystyle -\frac{i(\lambda+2\mu)}{\mu}[\xi^\alpha] \\[4mm]
            \displaystyle -\frac{i\mu}{\lambda+2\mu}[\xi_\beta] & -|\xi^{\prime}|
        \end{bmatrix}.
    \end{align}
\end{proposition}

\begin{theorem}[see \cite{TanLiu23}]\label{thm1.3}
    Let $(M,g)$ be a smooth compact Riemannian manifold of dimension $n$ with smooth boundary $\partial M$. Assume that the Lam\'{e} coefficients $\lambda,\,\mu \in C^{\infty}(\bar{M})$ satisfy $\mu>0$ and $\lambda + \mu \geqslant 0$. Let $\sigma(\Lambda_{\lambda,\mu}) \sim \sum_{j\leqslant 1} p_j(x,\xi^{\prime})$ be the full symbol of the elastic Dirichlet-to-Neumann map $\Lambda_{\lambda,\mu}$. Then, in boundary normal coordinates,
    \begin{align}
        \label{18} p_1(x,\xi^{\prime})&=
        \begin{bmatrix}
            \displaystyle \mu|\xi^{\prime}|I_{n-1} + \frac{\mu(\lambda+\mu)}{(\lambda+3\mu)|\xi^{\prime}|} [\xi^\alpha\xi_\beta] & \displaystyle -\frac{2i\mu^2}{\lambda+3\mu} [\xi^\alpha]\\[4mm]
            \displaystyle \frac{2i\mu^2}{\lambda+3\mu}[\xi_\beta] &\displaystyle  \frac{2\mu(\lambda+2\mu)}{\lambda+3\mu}|\xi^{\prime}|
        \end{bmatrix},\\
        \label{19} p_0(x,\xi^{\prime})&=
        \begin{bmatrix}
            \mu I_{n-1} &0  \\
            0& \lambda+2\mu 
        \end{bmatrix}
        q_0(x,\xi^{\prime}) -
        \begin{bmatrix}
            0 & 0 \\
            \lambda [\Gamma^\alpha_{\alpha\beta}] & \lambda \Gamma^\alpha_{\alpha n} 
        \end{bmatrix},\\
        \label{20} p_{-m}(x,\xi^{\prime})&=
        \begin{bmatrix}
            \mu I_{n-1} &0  \\
            0& \lambda+2\mu 
        \end{bmatrix}
        q_{-m}(x,\xi^{\prime}),\quad m\geqslant 1,
    \end{align}
    where $q_{-m}(x,\xi^{\prime})$, $m\geqslant 0$, are the remain symbols of a pseudodifferential operator $($see \eqref{3.1.1}$)$.
\end{theorem}

\vspace{5mm}

\section{Determining the Riemannian metric on the boundary}\label{s3}

\vspace{5mm}

\begin{proof}[Proof of Theorem {\rm \ref{thm1.1}}]
It follows from \eqref{18} that the $(n,n)$-entry $(p_1)^{n}_{n}$ of $p_1$ is
\begin{align*}
    (p_1)^{n}_{n} 
    = \frac{2\mu(\lambda+2\mu)}{\lambda+3\mu} |\xi^{\prime}|
    = \frac{2\mu(\lambda+2\mu)}{\lambda+3\mu} \sqrt{g^{\alpha\beta}\xi_\alpha\xi_\beta}.
\end{align*}
Since the Lam\'{e} coefficients $\lambda,\mu$ are known, this shows that $p_1$ uniquely determines $g^{\alpha\beta}$ on the boundary $\partial M$ for all $1\leqslant\alpha,\beta\leqslant n-1$. Clearly, the tangential derivatives $\frac{\partial g^{\alpha\beta}}{\partial x_\gamma}$ on the boundary $\partial M$ for all $1\leqslant \gamma \leqslant n-1$ can also be uniquely determined by $p_1$.

For $k\geqslant 0$, we denote by $T_{-k}$ the terms that only involve the boundary values of $g_{\alpha\beta}$, $g^{\alpha\beta}$, and their normal derivatives of order ar most $k$. Note that $T_{-k}$ may be different in different expressions. From \eqref{5.10}, we have
\begin{align}\label{3.06}
    b_0=\Gamma^{\alpha}_{n\alpha}I_{n} +T_0.
\end{align}
It follows from \eqref{3.9} and \eqref{5.24} that
\begin{align*}
    q_1=|\xi^{\prime}|I_{n} + s_2F_1,
\end{align*}
where $s_2=\frac{\lambda+\mu}{\lambda+3\mu}$. We compute that
\begin{align*}
    b_0 q_1 &=\Gamma^{\alpha}_{n\alpha}|\xi^{\prime}|I_{n} + s_2\Gamma^{\alpha}_{n\alpha}F_1 + 
    \begin{bmatrix}
        2|\xi^{\prime}|[\Gamma^\alpha_{\beta n}] & 0 \\[2mm]
        0 & 0
    \end{bmatrix}\notag\\
    &\quad + s_2
    \begin{bmatrix}
        2|\xi^{\prime}|^{-1}[\Gamma^\alpha_{\gamma n} \xi^{\gamma}\xi_{\beta}] & 2i[\Gamma^\alpha_{\gamma n} \xi^{\gamma}] \\[2mm]
        0 & 0
    \end{bmatrix}
    +T_0,\notag\\
    \frac{\partial q_1}{\partial x_n} &=\frac{\partial |\xi^{\prime}|}{\partial x_n}I_{n} + s_2
    \begin{bmatrix}
        \frac{\partial |\xi^{\prime}|^{-1}}{\partial x_n}[\xi^{\alpha}\xi_{\beta}] + |\xi^{\prime}|^{-1}\Big[\frac{\partial \xi^{\alpha}}{\partial x_n}\xi_{\beta}\Big] & i\Big[\frac{\partial \xi^{\alpha}}{\partial x_n}\Big] \\[4mm]
        0 & -\frac{\partial |\xi^{\prime}|}{\partial x_n}
    \end{bmatrix}
    +T_0,\notag\\
    -c_1 &= -\frac{i(\lambda+\mu)}{\mu}
    \begin{bmatrix}
        0 & \Gamma^\beta_{\beta n} [\xi^{\alpha}]\\[2mm]
        0 &0
    \end{bmatrix}
    -2i
    \begin{bmatrix}
        0 & [\xi^{\gamma}\Gamma^{\alpha}_{\gamma n}]\\[2mm]
        \frac{\mu}{\lambda+2\mu}[\xi^{\gamma}\Gamma^{n}_{\gamma\beta}] & 0
    \end{bmatrix}
    +T_0.
\end{align*}

By \eqref{5.24}, \eqref{5.25} and \eqref{2.11}, we get that the $(n,n)$-entries $(F_2E_1)^n_n$ and $(E_1F_1)^n_n$ are, respectively,
\begin{align}
    \label{6.1} (F_2E_1)^n_n &= -\frac{i\mu}{\lambda+2\mu}\xi_{\beta}\Bigl( is_2\Gamma^{\alpha}_{n\alpha}\xi^{\beta}+2is_2\Gamma^{\beta}_{n\gamma}\xi^{\gamma}+is_2\frac{\partial \xi^{\beta}}{\partial x_n}-\frac{i(\lambda+\mu)}{\mu}\Gamma^\gamma_{\gamma n} \xi^{\beta}-2i\xi^{\gamma}\Gamma^{\beta}_{\gamma n} \Bigr) \notag\\
    &\quad - |\xi^{\prime}|\Bigl( \Gamma^{\alpha}_{n\alpha}|\xi^{\prime}|-s_2\Gamma^{\alpha}_{n\alpha}|\xi^{\prime}|+\frac{\partial |\xi^{\prime}|}{\partial x_n}-s_2\frac{\partial |\xi^{\prime}|}{\partial x_n} \Bigr)+T_0\notag\\
    &= -\Gamma^{\alpha}_{n\alpha}|\xi^{\prime}|^2+\frac{2\mu(s_2-1)}{\lambda+2\mu}\Gamma^{\beta}_{\gamma n}\xi^{\gamma}\xi_{\beta}+\Bigl(\frac{\mu}{\lambda+2\mu}s_2-\frac{1}{2}+\frac{1}{2}s_2\Bigr)\frac{\partial |\xi^{\prime}|^2}{\partial x_n}+T_0,\\
    \label{6.2} (E_1F_1)^n_n & = i\xi^{\beta}\Bigl(is_2\Gamma^{\alpha}_{n\alpha}\xi_{\beta}-\frac{2i\mu}{\lambda+2\mu}\xi^{\gamma}\Gamma^n_{\gamma\beta}\Bigr) - |\xi^{\prime}|\Bigl( \Gamma^{\alpha}_{n\alpha}|\xi^{\prime}|-s_2\Gamma^{\alpha}_{n\alpha}|\xi^{\prime}|+\frac{\partial |\xi^{\prime}|}{\partial x_n}-s_2\frac{\partial |\xi^{\prime}|}{\partial x_n} \Bigr)\notag\\
    &\quad +T_0\notag\\
    &=-\Gamma^{\alpha}_{n\alpha}|\xi^{\prime}|^2+\frac{2\mu}{\lambda+2\mu}\Gamma^{n}_{\gamma\beta}\xi^{\gamma}\xi^{\beta}+\frac{1}{2}(s_2-1)\frac{\partial |\xi^{\prime}|^2}{\partial x_n}+T_0.
\end{align}
The $(n,\beta)$-entry $(F_2E_1)^n_\beta$ is 
\begin{align}
    \label{6.3} (F_2E_1)^n_\beta &= -\frac{i\mu}{\lambda+2\mu}\xi_{\alpha}\Bigl[ \Gamma^{\gamma}_{n\gamma}|\xi^{\prime}|\delta_{\alpha\beta}+s_2\Gamma^{\gamma}_{n\gamma}|\xi^{\prime}|^{-1}\xi^\alpha\xi_\beta+2|\xi^{\prime}|\Gamma^{\alpha}_{\beta n}+2s_2\Gamma^{\alpha}_{\gamma n}\xi^\gamma\xi_\beta|\xi^{\prime}|^{-1} \notag\\
    &\ + \frac{\partial |\xi^{\prime}|}{\partial x_n}\delta_{\alpha\beta}+s_2\Bigl( \frac{\partial |\xi^{\prime}|^{-1}}{\partial x_n}\xi^\alpha\xi_\beta+ |\xi^{\prime}|^{-1}\frac{\partial \xi^{\alpha}}{\partial x_n}\xi_\beta\Bigr)\Bigr] -|\xi^{\prime}|\Bigl(is_2\Gamma^{\alpha}_{n\alpha}\xi_\beta - \frac{2i\mu}{\lambda+2\mu}\Gamma^{n}_{\gamma\beta}\xi^{\gamma}\Bigr)\notag\\
    &\ +T_0\notag\\
    &=\Bigl(-\frac{i\mu(s_2+1)}{\lambda+2\mu}-is_2\Bigr)\Gamma^{\alpha}_{n\alpha}|\xi^{\prime}|\xi_\beta-\frac{2i\mu}{\lambda+2\mu}\Gamma^{\alpha}_{n\beta}|\xi^{\prime}|\xi_\alpha-\frac{2i\mu s_2}{\lambda+2\mu}\Gamma^{\alpha}_{n\gamma}|\xi^{\prime}|^{-1}\xi^\gamma\xi_\alpha\xi_\beta \notag\\
    &\quad -\frac{i\mu(s_2+1)}{\lambda+2\mu}\frac{\partial |\xi^{\prime}|}{\partial x_n}\xi_\beta+\frac{2i\mu}{\lambda+2\mu}\Gamma^{n}_{\gamma\beta}|\xi^{\prime}|\xi^\gamma+T_0.
\end{align}
Hence, the $(n,n)$-entry $(F_2E_1F_1)_{n}^n$ is 
\begin{align}
    \label{6.4} &(F_2E_1F_1)_{n}^n =(F_2E_1)^n_\beta (F_1)^\beta_n+(F_2E_1)^n_n(F_1)^n_n\notag\\
    & = i\xi^\beta \Bigl[\Bigl(-\frac{i\mu(s_2+1)}{\lambda+2\mu}-is_2\Bigr)\Gamma^{\alpha}_{n\alpha}|\xi^{\prime}|\xi_\beta-\frac{2i\mu}{\lambda+2\mu}\Gamma^{\alpha}_{n\beta}|\xi^{\prime}|\xi_\alpha -\frac{2i\mu s_2}{\lambda+2\mu}\Gamma^{\alpha}_{n\gamma}|\xi^{\prime}|^{-1}\xi^\gamma\xi_\alpha\xi_\beta\notag\\
    &\quad -\frac{i\mu(s_2+1)}{\lambda+2\mu}\frac{\partial |\xi^{\prime}|}{\partial x_n}\xi_\beta
     +\frac{2i\mu}{\lambda+2\mu}\Gamma^{n}_{\gamma\beta}|\xi^{\prime}|\xi^\gamma\Bigr]-|\xi^{\prime}| \Bigl[-\Gamma^{\alpha}_{n\alpha}|\xi^{\prime}|^2+\frac{2\mu(s_2-1)}{\lambda+2\mu}\Gamma^{\beta}_{\gamma n}\xi^{\gamma}\xi_{\beta} \notag\\
     &\quad +\Bigl(\frac{\mu s_2}{\lambda+2\mu}-\frac{1}{2}+\frac{1}{2}s_2\Bigr)\frac{\partial |\xi^{\prime}|^2}{\partial x_n}\Bigr]+T_0\notag\\
     &= 2\Gamma^{\alpha}_{n\alpha}|\xi^{\prime}|^3+\frac{4\mu}{\lambda+2\mu}\Gamma^{\alpha}_{n\gamma}\xi^\gamma\xi_\alpha|\xi^{\prime}|-\frac{2\mu}{\lambda+2\mu}\Gamma^{n}_{\beta\gamma}\xi^\gamma\xi^\beta|\xi^{\prime}|+\frac{\mu}{\lambda+2\mu}|\xi^{\prime}|\frac{\partial |\xi^{\prime}|^2}{\partial x_n}+T_0.
\end{align}
Therefore, by \eqref{3.1.1}, we obtain that the $(n,n)$-entry $(q_0)_{n}^n$ is 
\begin{align}
    \label{6.5} (q_0)_{n}^n &= \frac{1}{2}(1+s_2^2)\Gamma^{\alpha}_{n\alpha}+\frac{1}{4}(1-s_2^2)|\xi^{\prime}|^{-2}\frac{\partial |\xi^{\prime}|^2}{\partial x_n}\notag\\
    &\quad +\frac{\mu s_2}{\lambda+3\mu}(\Gamma^{\beta}_{\gamma n}\xi^{\gamma}\xi_\beta-\Gamma^{n}_{\beta\gamma}\xi^\gamma\xi^\beta)|\xi^{\prime}|^{-2}+T_0.
\end{align}
Note that in boundary normal coordinates, by computing directly, we have
\begin{align*}
    \Gamma^{\beta}_{\gamma n}\xi^{\gamma}\xi_\beta=-\Gamma^{n}_{\beta\gamma}\xi^\gamma\xi^\beta=-\frac{1}{2}\frac{\partial g^{\beta\gamma}}{\partial x_n}\xi_\beta\xi_\gamma=-\frac{1}{2}\frac{\partial |\xi^{\prime}|^2}{\partial x_n}.
\end{align*}
Substituting the above equalities and $s_2=\frac{\lambda+\mu}{\lambda+3\mu}$ into \eqref{6.5}, we get
\begin{align}
    \label{6.6} (q_0)_{n}^n &= \frac{\lambda^2+4\lambda\mu+5\mu^2}{(\lambda+3\mu)^2}\Gamma^{\alpha}_{n\alpha}+\frac{\mu^2}{(\lambda+3\mu)^2}|\xi^{\prime}|^{-2}\frac{\partial |\xi^{\prime}|^2}{\partial x_n}+T_0.
\end{align}
By \eqref{19}, we obtain
\begin{align}
    \label{6.7} (p_0)_{n}^n
    &=(\lambda+2\mu)(q_0)_{n}^n-\lambda \Gamma^{\alpha}_{n\alpha} \notag\\
    &=\frac{2\mu^2(2\lambda+5\mu)}{(\lambda+3\mu)^2}\Gamma^{\alpha}_{n\alpha}+\frac{\mu^2(\lambda+2\mu)}{(\lambda+3\mu)^2}|\xi^{\prime}|^{-2}\frac{\partial |\xi^{\prime}|^2}{\partial x_n}+T_0.
\end{align}
In boundary normal coordinates, we have
\begin{align*}
    \Gamma^{\alpha}_{n\alpha}=\frac{1}{2}g^{\alpha\beta}\frac{\partial g_{\alpha\beta}}{\partial x_n}=-\frac{1}{2}g_{\alpha\beta}\frac{\partial g^{\alpha\beta}}{\partial x_n}.
\end{align*}
Substituting this into \eqref{6.7}, we get
\begin{align}
    \label{6.8} (p_0)_{n}^n=-\frac{\mu^2}{(\lambda+3\mu)^2|\xi^{\prime}|^2}k_1^{\alpha\beta}\xi_\alpha\xi_\beta+T_0,
\end{align}
where
\begin{align}
    \label{6.9} k_1^{\alpha\beta}&=(2\lambda+5\mu)h_1g^{\alpha\beta}-(\lambda+2\mu)\frac{\partial g^{\alpha\beta}}{\partial x_n},\\
    \label{6.10} h_1&=g_{\alpha\beta}\frac{\partial g^{\alpha\beta}}{\partial x_n}.
\end{align}
Evaluating $(p_0)_{n}^n$ on all unit vectors $\xi^{\prime}$ shows that $p_0$ and the values of $g^{\alpha\beta}$ on the boundary $\partial M$ completely determine $k_1^{\alpha\beta}$. By \eqref{6.9} and \eqref{6.10}, we have 
\begin{align*}
    k_1^{\alpha\beta}g_{\alpha\beta}
    &=\bigl((n-1)(2\lambda+5\mu)-(\lambda+2\mu)\bigr)h_1.
\end{align*}
It follows from \eqref{0.3} that, for $n\geqslant 2$,
\begin{align*}
    (n-1)(2\lambda+5\mu)-(\lambda+2\mu)=(2n-3)(\lambda+\mu)+(3n-4)\mu >0.
\end{align*}
Then,
\begin{align*}
    h_1=\frac{k_1^{\alpha\beta}g_{\alpha\beta}}{(n-1)(2\lambda+5\mu)-(\lambda+2\mu)}.
\end{align*}
By \eqref{6.9}, we get that
\begin{align*}
    \frac{\partial g^{\alpha\beta}}{\partial x_n}=\frac{(2\lambda+5\mu)h_1g^{\alpha\beta}-k_1^{\alpha\beta}}{\lambda+2\mu},
\end{align*}
which implies that $p_0$ uniquely determines $\frac{\partial g^{\alpha\beta}}{\partial x_n}$ on the boundary $\partial M$.

According to \eqref{20}, \eqref{3.1.1}, and \eqref{3.05}, we know that $p_{-1}$ uniquely determines $q_{-1}$, $q_{-1}$
uniquely determines $E_0$, and
\begin{align}
    \label{6.11} (E_0)^n_n=\Bigl(\frac{\partial q_0}{\partial x_n}\Bigr)^n_n - (c_0)^n_n + T_{-1}.
\end{align}
In view of that
\begin{align*}
    \frac{\partial^2 (g_{\alpha\beta}g^{\alpha\beta})}{\partial x_n^2}=\frac{\partial^2 (n-1)}{\partial x_n^2}=0,
\end{align*}
we get
\begin{align*}
    g_{\alpha\beta}\frac{\partial^2 g^{\alpha\beta}}{\partial x_n^2} = - g^{\alpha\beta}\frac{\partial^2 g_{\alpha\beta}}{\partial x_n^2}+T_{-1}.
\end{align*}
Then, in boundary normal coordinates, we compute that
\begin{align*}
    \frac{\partial \Gamma^{\alpha}_{n\alpha}}{\partial x_n}&=-\frac{1}{2}g_{\alpha\beta}\frac{\partial^2 g^{\alpha\beta}}{\partial x_n^2}+T_{-1},\\
    g^{ml}\frac{\partial \Gamma^n_{ml}}{\partial x_n}&=\frac{1}{2}g_{\alpha\beta}\frac{\partial^2 g^{\alpha\beta}}{\partial x_n^2}+T_{-1}.
\end{align*}
Combining \eqref{6.11}, \eqref{6.6}, \eqref{5.13}, and the above results, we have 
\begin{align}
    \label{6.12} (E_0)^n_n 
    &= \frac{\lambda^2+4\lambda\mu+5\mu^2}{(\lambda+3\mu)^2} \frac{\partial \Gamma^{\alpha}_{n\alpha}}{\partial x_n} +\frac{\mu^2}{(\lambda+3\mu)^2}|\xi^{\prime}|^{-2}\frac{\partial^2 |\xi^{\prime}|^2}{\partial x_n^2} - \frac{\lambda+\mu}{\lambda+2\mu}\frac{\partial \Gamma^\alpha_{\alpha n}}{\partial x_n}  \notag\\
    &\quad - \frac{\mu}{\lambda+2\mu}g^{ml}\frac{\partial \Gamma^n_{ml}}{\partial x_n} +T_{-1} \notag\\
    &= -\frac{\mu^2(2\lambda+5\mu)}{(\lambda+3\mu)^2(\lambda+2\mu)}g_{\alpha\beta}\frac{\partial^2 g^{\alpha\beta}}{\partial x_n^2}+\frac{\mu^2}{(\lambda+3\mu)^2}|\xi^{\prime}|^{-2}\frac{\partial^2 g^{\alpha\beta}}{\partial x_n^2}\xi_\alpha\xi_\beta+T_{-1} \notag\\
    &= -\frac{\mu^2}{(\lambda+3\mu)^2(\lambda+2\mu)|\xi^{\prime}|^{2}}k_2^{\alpha\beta}\xi_\alpha\xi_\beta+T_{-1},
\end{align}
where
\begin{align}
    \label{6.13} k_2^{\alpha\beta}&=(2\lambda+5\mu)h_2g^{\alpha\beta}-(\lambda+2\mu)\frac{\partial^2 g^{\alpha\beta}}{\partial x_n^2},\\
    \label{6.14} h_2&=g_{\alpha\beta}\frac{\partial^2 g^{\alpha\beta}}{\partial x_n^2}.
\end{align}
By the same argument, we prove that $p_{-1}$ uniquely determines $\frac{\partial^2 g^{\alpha\beta}}{\partial x_n^2}$ on the boundary $\partial M$.

Now we consider $p_{-m-1}$ for $m\geqslant 1$. By \eqref{20} and \eqref{3.1.1}, we have $p_{-m-1}$ uniquely determines $q_{-m-1}$, and $E_{-m}$ can be determined from the knowledge of $q_{-m-1}$. From \eqref{4.1}, we see that
\begin{align}
    \label{6.18} E_{-m}=\frac{\partial q_{-m}}{\partial x_n}+T_{-m-1}.
\end{align}
We end this proof by induction. Suppose we have shown that
\begin{align}
    \label{6.15} (E_{-j})^n_n 
    &= -\frac{\mu^2}{(\lambda+3\mu)^2(\lambda+2\mu)|\xi^{\prime}|^{2}}k_{j+2}^{\alpha\beta}\xi_\alpha\xi_\beta+T_{-j-1}
\end{align}
for $1\leqslant j \leqslant m$, where
\begin{align}
    \label{6.16} k_{j+2}^{\alpha\beta}&=(2\lambda+5\mu)h_{j+2}g^{\alpha\beta}-(\lambda+2\mu)\frac{\partial^{j+2} g^{\alpha\beta}}{\partial x_n^{j+2}},\\
    \label{6.17} h_{j+2}&=g_{\alpha\beta}\frac{\partial^{j+2} g^{\alpha\beta}}{\partial x_n^{j+2}}.
\end{align}
This means that $p_{-j-1}$ uniquely determines $\frac{\partial^{j+2} g^{\alpha\beta}}{\partial x_n^{j+2}}$ on the boundary $\partial M$ for $1\leqslant j \leqslant m$.

Since we have $p_{-(m+1)-1}$ uniquely determines $q_{-(m+1)-1}$, and $q_{-(m+1)-1}$ uniquely determines $E_{-(m+1)}$. From \eqref{6.18}, we have
\begin{align*}
    E_{-(m+1)}=\frac{\partial q_{-(m+1)}}{\partial x_n}+T_{-(m+1)-1}.
\end{align*}
By the above equality and the fact that $q_{-(m+1)}$ uniquely determines $E_{-m}$, we have $E_{-(m+1)}$ uniquely determines $\frac{\partial E_{-m}}{\partial x_n}$. By the assumption \eqref{6.15}, we get 
\begin{align*}
    \Big(\frac{\partial E_{-m}}{\partial x_n}\Big)^n_n 
    &= -\frac{\mu^2}{(\lambda+3\mu)^2(\lambda+2\mu)|\xi^{\prime}|^{2}}k_{m+3}^{\alpha\beta}\xi_\alpha\xi_\beta+T_{-j-2},
\end{align*}
where
\begin{align*}
    k_{m+3}^{\alpha\beta}&=(2\lambda+5\mu)h_{m+3}g^{\alpha\beta}-(\lambda+2\mu)\frac{\partial^{m+3} g^{\alpha\beta}}{\partial x_n^{m+3}},\\
    h_{m+3}&=g_{\alpha\beta}\frac{\partial^{m+3} g^{\alpha\beta}}{\partial x_n^{m+3}}.
\end{align*}
By the same argument, we see that $p_{-(m+1)-1}$ uniquely determines $\frac{\partial^{m+3} g^{\alpha\beta}}{\partial x_n^{m+3}}$ on the boundary $\partial M$. Therefore, we conclude that the elastic Dirichlet-to-Neumann map $\Lambda_{\lambda,\mu}$ uniquely determines the partial derivatives of all orders of the Riemannian metric $\frac{\partial^{|J|} g^{\alpha\beta}}{\partial x^J}$ on the boundary $\partial M$ for all multi-indices $J$.

\end{proof}

\vspace{5mm}

\section*{Acknowledgements}

\vspace{5mm}

This work was supported by National Key R\&D Program of China 2020YFA0712800.

\vspace{5mm}

\end{document}